\documentclass[11pt]{amsart}
\usepackage{amsmath}
\usepackage{amssymb}
\usepackage{esint}
\usepackage{tikz}
\usepackage{comment}

\theoremstyle{plain}
\newtheorem{theorem}{Theorem}[section]
\newtheorem{lemma}[theorem]{Lemma}

\newtheorem{cor}[theorem]{Corollary}

\theoremstyle{definition}
\newtheorem{remark}[theorem]{Remark}

\newtheorem{example}[theorem]{Example}
\numberwithin{equation}{section}
\newtheorem*{theorem*}{Theorem A}

\def\be{\begin{equation}}
\def\ee{\end{equation}}

\begin{document}

\title[The Monge-Amp\`ere equation in cones and polygonal domains ]
{ The Monge-Amp\`ere equation in cones and polygonal domains }
\author[Huang]{Genggeng Huang}
\address{School of Mathematical Sciences\\ Fudan University\\Shanghai, China} \email{genggenghuang@fudan.edu.cn }

\author[Shen]{Weiming Shen}
\address{School of Mathematical Sciences\\
Capital Normal University\\
Beijing, 100048, China}
\email{wmshen@aliyun.com}

\begin{abstract}
We study asymptotic behaviors of solutions to the Monge-Amp\`ere equation in cones and use the expansion as a tool to study the regularity of solutions
in polygonal domains.
\end{abstract}

\thanks{The first author is partially supported by NSFC 12141105. The second author is partially supported by NSFC 12371208.
}

\maketitle


\section{Introduction}\label{sec-Intro}
In this paper, we study the asymptotic behaviours of the solutions for the Dirichlet problem of the Monge-Amp\`ere equation near the vertices of the convex polygons in $\mathbb R^2$.

The Dirichlet problem for Monge-Amp\`ere equation, i.e the following problem
\begin{equation}\label{intro-1}
\begin{split}
&\det D^2 u=f,\quad \text{in}\quad \Omega,\\
&u=\varphi,\quad \text{on}\quad \partial\Omega
\end{split}
\end{equation}
is a classical problem in the study of Monge-Amp\`ere equation. When $f\ge c_0>0$ in $\Omega$, \eqref{intro-1} is the non-degenerate case. The global $C^{2,\alpha}$-estimate is important for the global smoothness of the solution $u$. Since for $\varphi=0$, the eigenvalues of the second order tangential derivatives of $u$ are proportional to the principle eigenvalues of $\partial\Omega$. It is natural to assume that $\Omega$ is smooth and uniformly convex.  For $f\in C^{1,1}(\overline\Omega)$, the global $C^{2,\alpha}$ estimates are obtained in the works Ivockina \cite{Iv1980}, Krylov \cite{Kry1983} and Caffarelli-Nirenberg-Spruck \cite{CaffarelliNirenbergSpruck1984}. Since the above works need to differentiate \eqref{intro-1} twice to get the estimate $\|u\|_{C^{1,1}(\overline{\Omega})}$,  the condition $f\in C^{1,1}(\overline\Omega)$ can't be removed. People are also interested in the sharp conditions on the right-hand side and boundary data for global $C^{2,\alpha}$ estimates.  Wang \cite{Wang1996} made a first trial to relax the conditions to $f\in C^{0,1}(\overline{\Omega})$ and $\partial\Omega\in C^3$, $\varphi\in C^{3}(\overline{\Omega})$. In \cite{Wang1996}, he also constructed examples to show that $\partial\Omega\in C^3$, $\varphi\in C^{3}(\overline{\Omega})$ are optimal. For uniformly convex domain, under the sharp conditions, global $C^{2,\alpha}$ estimates were obtained by Trudinger-Wang\cite{TW2008}. For smooth convex domain, under the sharp conditions, pointwise $C^{2,\alpha}$ estimates up to the boundary were proved by Savin \cite{Savin2013}. For general smooth bounded domain, Guan-Spruck\cite{GuanSpruck1993} first obtained the global $C^{2,\alpha}$ estimate by assuming the existence of a convex strict subsolution $\underline{u}\in C^2(\overline{\Omega})$. The strictness of subsolution $\underline u$ was removed by Guan\cite{Guan1998}.

However, the global $C^{2,\alpha}$ estimates for all the above works are under the assumptions that the domains $\Omega$ are smooth enough. Until recently, Le-Savin \cite{LeSavin2023} made a first attempt on the global $C^{2,\alpha}$ estimates in two dimensional polygons.
They proved that\\
\begin{theorem*} [Theorem 1.1,\cite{LeSavin2023}]  Let $\Omega$ be a bounded convex polygonal domain in $\mathbb R^2$. Let $u$ be a convex function that solves the Dirichlet problem for the Monge-Amp\`ere equation \eqref{intro-1}. Assume that for some $\beta\in(0,1)$,
\begin{equation*}
f\in C^\beta(\overline{\Omega}),\quad f>0,\quad \varphi\in C^{2,\beta}(\partial\Omega)
\end{equation*}
and there exists a globally $C^2$, convex, strictly sub-solution $\underline u\in C^2(\overline{\Omega})$ to \eqref{intro-1}. Then $u\in C^{2,\alpha}(\overline\Omega)$ for some $\alpha>0$. The constant $\alpha$ and the global $C^{2,\alpha}$ norm $\|u\|_{C^{2,\alpha}(\overline{\Omega})}$ depend on $\Omega$, $\beta$, $\min_{\overline\Omega}f$, $\|f\|_{C^\beta(\overline{\Omega})}$, $\|\varphi\|_{C^{2,\beta}(\partial\Omega)}$, $\|\underline u\|_{C^2(\overline\Omega)}$ and the differences $\det D^2{\underline u}-f$ at the vertices of $\Omega$.
\end{theorem*}
To the authors' knowledge, \cite{LeSavin2023} is the first result on the global  $C^{2,\alpha}$ estimates in the polygons for Monge-Amp{\`e}re equations. However, the higher regularity for the solutions of the Monge-Amp\`ere equation remains unknown in the polygons. There is another boundary value problem for Monge-Amp\`ere equation in polygonal domains.  Rubin \cite{Rubin2015} consider the following Guillemin boundary value problem
\begin{eqnarray}
&& \det D^2 u=\frac{\eta(x)}{L(x)},\quad \text{in}\quad P,\label{MAG01}\\
&& u(x)-\sum_{i=1}^N l_i(x)\ln l_i(x)\in C^\infty(\overline{P})\label{MAG02}
\end{eqnarray}
where $P\subset \mathbb R^2$ is the convex polygon
\begin{equation}\label{polygon1}
P=\bigcap_{i=1}^N \{l_i(x)>0\},
\end{equation}
where $l_i(x)$ is an affine function of $x$ for each $i$.
And
\begin{equation*}
L(x)=\prod_{i=1}^N l_i(x),\quad 0<\eta(x)\in C^\infty(\overline P).
\end{equation*}
Rubin \cite{Rubin2015} proved the existence of solutions for \eqref{MAG01} coupled with the following boundary condition
\begin{equation}\label{MAG03}
u(x)-\sum_{i=1}^N l_i(x)\ln l_i(x)\in C^\infty(\overline{P}\backslash\{p_1,\cdots,p_N\})\cap C^\alpha(\overline P).
\end{equation}
Here $p_1,\cdots,p_N$ are vertices of the polygon $P$. The global smoothness for $v(x)=u(x)-\sum_{i=1}^N l_i(x)\ln l_i(x)$ was obtained recently by Huang \cite{Huang2023} in $P\subset\mathbb R^2$ and Huang-Shen \cite{HuangShen2023} for higher dimensional cases.

\par This leads the authors consider the higher regularity of the solutions for \eqref{intro-1} in the polygon. However, the problem is different from the Guillemin boundary value problem. There are obstructions for the global smoothness of the solutions even after subtracting finite singular terms.  This seems to be a common phenomenon in the boundary expansion for singular elliptic equations. See \cite{HanJiang2023,HanLiLi2023,HanShen2020}. The first main result in this paper is the  higher global regularity of the solutions of \eqref{intro-1}.
For simplicity, we consider  the asymptotic behaviour near the conical point in a rather simple model problem.

Let $\mu\in(0,1)$, $V_{\mu}\subseteq  \mathbb{R}^2$ be an infinity cone with its vertex at the origin and an open angle $\mu\pi$.
 Using polar coordinates, we write
$$V_{\mu}=\{(r,\theta)| r\in (0,\infty),\, \theta\in (0,\mu\pi)\}.$$

Let $u$ be a solution of
\begin{align}\label{eq-MA-cone-eq}
\det D^2 u&=1
\quad\quad\quad\text{   in } V_{\mu}\bigcap B_1(0),\\
\label{eq-MA-cone-boundary}u&=\frac{|x|^2}{2}  \quad\quad\text{on }\partial V_{\mu} \bigcap B_1(0).
\end{align}
Set $u=\frac{|x|^2}{2}+v$, then
\begin{equation}\label{eq-v}\Delta v= -\det D^2 v:= F_0(v).\end{equation}
Assume
\begin{equation}\label{v-setup1}
v\in C^{2,\alpha}(\overline{V_{\mu}\bigcap B_1(0)}) \quad\text{for some }\alpha\in(0,1),
\end{equation}
and
\begin{equation}\label{v-setup2}
D^2v(0)=0.
\end{equation}
We will study asymptotic behaviors of $u$ or $v$ near the vertex of the cone $V_{\mu}$.

We introduce the {\it index set} $\mathcal I_1$, $\mathcal I_2$ and $\mathcal I$ by defining
\begin{align}\label{eq-def-index-1}
\mathcal I= \mathcal I_1 \bigcup\mathcal I_2,\quad \mathcal I_1= \left\{\frac{i}{\mu}\big|i\ge 1\right\},
\end{align}
\begin{align}\label{eq-def-index-2}\begin{split}
\mathcal I_2  &=\left\{\left(\frac 1\mu-2\right)j+\frac{i}{\mu}\big{|}i\ge 1,\  j\ge 1 \right\}, \quad\text{if  } \mu \in (0,1/2),\\
&=\left\{ 2\left(\frac 1\mu-1\right)j+\frac{i}{\mu}\big{|}   i\ge 2,\  j\ge 1 \right\}, \quad\text{if  } \mu \in [1/2,1),
\end{split}
\end{align}

We note that the set $\mathcal I_1\bigcap \mathcal I_2$ is not necessarily an empty set.

We arrange $\mathcal I$ as follows:
\begin{equation}\label{eq-arrangement}
\left( \frac{1}{\mu}=\right)\,\mu_1<\mu_2<\mu_3 <\cdots<\mu_{k}<\cdots ,\quad\quad\rightarrow \infty
\end{equation}
If $0<\mu<1/2$, $\mu_2= \frac{2}{\mu}-2$ and $\mu_3=\min\{\frac{2}{\mu}, \frac{3}{\mu}-4\}$.

One of our goal in this paper is to investigate the asymptotic behavior of solution $u$ of \eqref{eq-MA-cone-eq}-\eqref{eq-MA-cone-boundary}
near the origin.


\begin{theorem}\label{thrm-expansion-1}
Let $\mu\in(0,1/2)$ and
$u$ be an solution of
\eqref{eq-MA-cone-eq}-\eqref{eq-MA-cone-boundary} with \eqref{v-setup1} and \eqref{v-setup2}.
Then, for any integer $m\ge 1$,
\begin{equation}\label{expansion-1}
   \Big|u(x)-\frac{|x|^2}{2}-\sum_{i=1}^{m}\sum_{j=0}^{i-1}c_{i,j}(\theta)|x|^{\mu_i}(-\ln{|x|})^j\Big |
     \leq C_m\theta(\mu\pi-\theta)|x|^{\mu_{m+1}}(-\ln{|x|})^m,
\end{equation}
where $C_m$ is a positive constant depending on $m,\mu$
and each $c_{i,j}\in C^{\infty}[0,\mu\pi]$  with
$c_{i,j}(0)=c_{i,j}(\mu\pi)=0$.
\end{theorem}
In particular, $c_{1,0}(\theta)=c\sin\frac{\theta}{\mu}$ for some constant $c$.
\begin{remark}
 Theorem \ref{thrm-expansion-1} can be applied to Theorem A. This implies that  the solutions have obstructions for higher regularity in Theorem A even for smooth $f$ and boundary data $\varphi$.
\end{remark}

\begin{remark}
The method in Theorem \ref{thrm-expansion-1} can also be used to study the following problem:
\begin{align*}
\det D^2 u&=f
\quad\quad\quad\text{   in } V_{\mu}\bigcap B_1(0),\\
u&=\frac{|x|^2}{2} +g \quad\quad\text{on }\partial V_{\mu} \bigcap B_1(0).
\end{align*}
with $v= u-\frac{|x|^2}{2}$ satisfying $
v\in C^{2,\alpha}(\overline{V_{\mu}\bigcap B_1(0)}) \quad\text{for some }\alpha\in(0,1)$ and
$D^2v(0)=0$, where $f \in C^{\infty}(\overline{V_{\mu}\bigcap B_1(0)})$ is a positive function with $f(0)=1$ and  $g \in C^{\infty}(\overline{V_{\mu}\bigcap B_1(0)})$ with
$g=O(|x|^3)$.

However, in this case, the calculations will be much more complicated.
\end{remark}

\begin{theorem}\label{thrm-expansion-2}
Let $\mu\in [1/2, 1)$ and
$u$ be an solution of \eqref{eq-MA-cone-eq}-\eqref{eq-MA-cone-boundary}
with \eqref{v-setup1} and \eqref{v-setup2}.
Then, for any integer $m\ge 2$,
\begin{equation}\label{expansion-2}
   \Big|u(x)-\frac{|x|^2}{2}-\sum_{i=2}^{m} \sum_{j=0}^{i-2} c_{i,j}(\theta)|x|^{\mu_i}(-\ln{|x|})^j\Big |
     \leq C_m \theta(\mu\pi-\theta)|x|^{\mu_{m+1}}(-\ln{|x|})^m,
\end{equation}
where $C_m$ is a positive constant depending on $m,\mu$
and each $c_{i,j}\in C^{\infty}[0,\mu\pi]$  with
$c_{i,j}(0)=c_{i,j}(\mu\pi)=0$.
\end{theorem}

We point out that the sequence $\{\mu_i\}$ in Theorem \ref{thrm-expansion-1} and Theorem \ref{thrm-expansion-2} is determined only by the open angle of the $V_{\mu}$, independent of the specific solution $u$ of
\eqref{eq-MA-cone-eq}-\eqref{eq-MA-cone-boundary}.

We will use the asymptotic expansion as a tool to study the regularity of solutions. Furthermore, in high-dimensional cases, we will construct examples where $C^2(\overline{\Omega})-$convex solutions to \eqref{intro-1} do not exist when \eqref{intro-1} satisfying both the compatibility conditions and the existence of a smooth subsolution.

\section{Asymptotic Expansions}\label{sec-expansion}

In order to prove Theorem \ref{thrm-expansion-1} and Theorem \ref{thrm-expansion-2}, we study the equation \eqref{eq-v} in cylindrical coordinates.
For any $x\in\mathbb R^2\setminus\{0\}$,
set $(t,\theta)\in \mathbb R\times \mathbb S^{1}$ by
\begin{equation}\label{eq-change-coordinates}
t=-\ln|x|,\quad \theta=\frac{x}{|x|}.\end{equation}

Denote $\partial^2_t+\partial^2_{\theta}$ by $\widetilde{\Delta}$. In the rest of this paper, we will use the superscript tilde to distinguish the representation of the same function in different coordinate systems. More precisely,
for the notation of a function in the $(t, \theta)$-coordinates, we will use its notation in the $x$-coordinates with the superscript tilde to represent it. For example, we will
denote the function $v(x_1,x_2) $ in the $(t, \theta)$-coordinates by $\widetilde{v} (t, \theta)$.
Then,
\begin{equation}\label{eq-in-t}\widetilde{\Delta}\widetilde{v}=e^{-2t}\widetilde{F_0(v)  }.\end{equation}
\eqref{v-setup1}-\eqref{v-setup2} imply, for any $t>10$,
\begin{equation}\label{v-est-t} \|\widetilde{v}\|_{C^{2, \alpha}([t-1,t+1]\times[0,\mu\pi])  }\leq Ce^{-(2+\alpha) t}.\end{equation}

In this setting, Theorem \ref{thrm-expansion-1} is equivalent to the following form.
\begin{theorem}\label{thm-v-1}
Let $\mu\in(0,1/2)$ and
$\widetilde{v}$ be a solution of
\eqref{eq-in-t} with \eqref{v-est-t}.
Then, for any integer $m\ge 1$,
\begin{equation}\label{add-1}
   \Big|\widetilde{v}-\sum_{i=1}^{m}\sum_{j=0}^{i-1}c_{i,j}(\theta)t^{j}e^{-\mu_it}\Big |
     \leq C_m \theta(\mu\pi-\theta)t^{m}e^{-\mu_{m+1}t}\quad\text{in } (100,\infty)\times[0,\mu\pi],
\end{equation}
where $C_m$ is a positive constant depending on $m,\mu$
and each $c_{i,j}\in C^{\infty}[0,\mu\pi]$  with
$c_{i,j}(0)=c_{i,j}(\mu\pi)=0$.
\end{theorem}

As a preparation,
we first discuss the linear equation $\widetilde{\Delta}\widetilde{w}=\widetilde{f}$ and then study
the nonlinear equation \eqref{eq-in-t}.

For $i\ge 1$, set $\phi_i=\sqrt{\frac{2}{\mu\pi}}\sin\frac{i\theta}{\mu}$. Then, $\{\phi_i\}$ is an orthonormal basis of $L^2[0,\mu\pi]$ and
$\frac{d^2\phi_i}{d \theta^2}=-\big(\frac{i}{\mu}\big)^2\phi_i$. For each $i\ge 1$ and any $\psi=\psi(t)\in C^2(\mathbb R)$, write
\begin{equation}\label{eq-U2-01a} \widetilde{\Delta}(\psi \phi_i)=(L_i\psi)\phi_i,\end{equation}
where $L_i$ is given by
\begin{equation}\label{eq-U2-01b}
L_i\psi=\psi_{tt}-\left(\frac{i}{\mu}\right)^2\psi.\end{equation}
Then,   kernel $\mathrm{Ker}(L_i)=\text{span}\left\{e^{-\frac i\mu t},e^{\frac i\mu t}\right\}$.

For each fixed $i\ge 1$, consider the linear equation
\begin{equation}\label{eq-linearization-ODE-i}
L_i\psi=\widetilde{f}\quad\text{on }(T,+\infty).
\end{equation}
For the solutions of \eqref{eq-linearization-ODE-i} which converge to zero as $t\to\infty$, we have the following lemma.

\begin{lemma}\label{lemma-decay-sol-Ljw=f-ODE}
Let
$\gamma>0$ be a constant,  $m\geq 0$ be an integer,
and $\widetilde{f}$ be a smooth function on $(T,+\infty)$, satisfying,
for any $t>T$,
$$|\widetilde{f}(t)|\leq C t^me^{-\gamma t}.$$
Let $\psi$ be a smooth solution of \eqref{eq-linearization-ODE-i} on $(T,+\infty)$  such that
$\psi(t)\to 0$ as $t\to +\infty$.
Then,
for any $t>T$,
\begin{equation*}
|\psi(t)|\leq\left\{
\begin{aligned}
&Ct^me^{-\gamma t}\hspace{1cm}\text{if } \gamma< \frac{i}{\mu},\\
&Ct^{m+1}e^{-\gamma t}\hspace{0.6cm}\text{if } \gamma= \frac{i}{\mu},
\end{aligned}
\right.
\end{equation*}
and there exists a constant $c$ such that
\begin{equation*}
|\psi(t)-ce^{-\frac{i}{\mu}t}|\leq
Ct^me^{-\gamma t}\hspace{0.6cm}\text{if } \gamma> \frac{i}{\mu}.
\end{equation*}
\end{lemma}
\begin{proof}
The proof is essentially the same as Lemma 7.1 in \cite{HJS}.
Set
\begin{align*}
\psi_*(t) &=-\frac{\mu e^{- \frac{i}{\mu}t}}{2i}\int_{T}^te^{ \frac{i}{\mu} s}\widetilde{f}(s)ds-
\frac{\mu e^{ \frac{i}{\mu}t}}{2i}\int_{t}^\infty e^{- \frac{i}{\mu}s}\widetilde{f}(s)ds, \quad\text{if  }  \gamma\le\frac{i}{\mu},\\
&=\frac{\mu e^{-\frac{i}{\mu} t}}{2i}\int_{t}^\infty e^{ \frac{i}{\mu}s}\widetilde{f}(s)ds-
\frac{\mu e^{ \frac{i}{\mu}t}}{2i}\int_{t}^\infty e^{- \frac{i}{\mu} s}\widetilde{f}(s)ds , \quad\text{if  } \gamma>\frac{i}{\mu}.
\end{align*}
Then, $\psi_*(t)$ is a particular solution satisfying the desired estimate.

Note that any solution $\psi$ of \eqref{eq-linearization-ODE-i} is given by
$$\psi=c_1e^{-\frac{i}{\mu}t}+c_2e^{\frac{i}{\mu}t}+\psi_*,$$
for some constants $c_1$ and $c_2$.
The condition $\psi(t)\to0$ as $t\to +\infty$ implies $c_2=0$.
Therefore, we have the desired estimate.
\end{proof}

Now, we study the linear equation
\begin{equation}\label{eq-linearization-PDE-spherical}    \widetilde{\Delta}\widetilde{w}=\widetilde{f}\quad\text{in }
(T,+\infty)\times [0,\mu\pi].\end{equation}
We discuss asymptotic expansions of solutions of \eqref{eq-linearization-PDE-spherical}
as $t\to \infty$.

\begin{lemma}\label{lemma-nonhomogeneous-linearized-eq}
Let
$\widetilde{f}$ be a smooth function in $(T,+\infty)\times[0,\mu\pi]$, satisfying,
for some positive constants $\gamma$, $A$, and some integer $m\geq 0$,
\begin{align}\label{eq-assumption-bound-f}
|\widetilde{f}|\le At^me^{-\gamma t}\quad\text{in } (T,+\infty)\times [0,\mu\pi],\end{align}
Let $\widetilde{w}\in C^\infty((T,+\infty)\times[0,\mu \pi])$
be a solution of \eqref{eq-linearization-PDE-spherical}
in $(T,+\infty)\times[0,\mu\pi]$ such that
$\widetilde{w}(t,0)=\widetilde{w}(t,\mu \pi)=0$ for $t\in(T,+\infty)$,
$\widetilde{w}(t,\theta)\to 0$ as $t\to +\infty$ uniformly for $\theta\in[0,\mu\pi]$.

$\mathrm{(i)}$ If $\gamma\le \frac{1}{\mu}$, then,
for any $(t,\theta)\in (T+1,+\infty)\times[0,\mu\pi]$,
$$|\widetilde{w}|\le \begin{cases} Ct^me^{-\gamma t} &\text{if }\gamma<\frac{1}{\mu},\\
Ct^{m+1}e^{-\gamma t}&\text{if }\gamma=\frac{1}{\mu}.\end{cases}$$

$\mathrm{(ii)}$ If $\frac{l}{\mu}<\gamma\le \frac{l+1}{\mu}$ for some positive integer $l$,
then, 
for any $(t,\theta)\in (T+1,+\infty)\times[0,\mu\pi]$,
$$\Big|\widetilde{w}-\sum_{i=1}^lc_ie^{-\frac{i}{\mu}t}\phi_i\Big|\le
\begin{cases} Ct^me^{-\gamma t}&\text{if }\frac{l}{\mu}<\gamma<\frac{l+1}{\mu},\\
Ct^{m+1}e^{-\gamma t}&\text{if }\gamma=\frac{l+1}{\mu},\end{cases}$$
for some constant $c_i$, for
$i=1, \cdots, l$.
\end{lemma}

\begin{proof} The proof
is similar as that of Lemma 7.2 in \cite{HJS}.
 Here, we only point out the key points.
 We consider (ii) only.

Set
$$\widetilde{w_{i}}(t)=\int_{[0,\mu\pi]}\widetilde{w}(t, \theta)\phi_{i}(\theta)d\theta,\quad
\widetilde{f_{i}}(t)=\int_{[0,\mu\pi] }\widetilde{f}(t, \theta)\phi_{i}(\theta)d\theta.$$
Then,
$L_i \widetilde{w_{i}}=\widetilde{f_{i}}\quad\text{on }(T,+\infty).$

By Lemma \ref{lemma-decay-sol-Ljw=f-ODE}, there exists a constant $c_{i}$ for $i=1, \cdots, l$ such that,
for any $t>T$,
 \begin{equation*}
|\widetilde{w_{i}}(t)-c_{i}e^{-\frac{i}{\mu} t}|\leq C_it^me^{-\gamma t}.
\end{equation*}

Set
$$\widehat w(t,\theta)=\widetilde{w}(t,\theta)
-\sum_{i=1}^{l}\widetilde{w_{i}}(t)\phi_{i}(\theta),\quad
\widehat f(t,\theta)=\widetilde{f}(t,\theta)
-\sum_{i=1}^{l}\widetilde{f_{i}}(t)\phi_{i}(\theta).$$
Then,  $\widetilde{w_{i}}(t)\to 0$ as $t\to +\infty$ and  $\widetilde{\Delta}\widehat{w}=\widehat{f}$.
Arguing as in \cite{HJS}, one gets

\begin{equation*}\|\widehat w(t,\cdot)\|_{L^2[0,\mu\pi]}\le Ct^{m_*}e^{-\gamma t},
\end{equation*}
where $m_*=m$ if $\frac{l}{\mu}<\gamma<\frac{l+1}{\mu}$ and $m_*=m+1$ if $\gamma=\frac{l+1}{\mu}$.

Applying the $W^{2,2}$-estimates, one gets
 for any $t>T+1$,
\begin{equation*}
\sup_{\{t\}\times[0,\mu\pi]} |\widehat w|\le
C\big\{\|\widehat w\|_{L^2((t-1,t+1)\times [0,\mu\pi])}+\|\widehat f\|_{L^{\infty}((t-1,t+1)\times[0,\mu\pi])}\big\}
\le Ct^{m_*}e^{-\gamma t}.\end{equation*}

Then, write
\begin{align*}
\tilde w(t,\theta)-\sum_{i=1}^lc_{i}e^{-\frac i\mu t}\phi_{i}(\theta)
=\sum_{i=1}^l\big(\tilde w_{i}(t)-c_{i}e^{-\frac i\mu t}\big)\phi_{i}(\theta)
+\widehat w(t,\theta).
\end{align*}
We finish the proof.

\end{proof}

\begin{lemma}\label{lemma-particular-solutions1}
Let $\gamma\notin\mathcal I_1 $ be a positive constant,
and $\widetilde{h}$ $\in C[0,\mu \pi]$.
Then, there exist $\widetilde{w} \in  C^2[0,\mu \pi]$ with $\widetilde{w}(0)=\widetilde{w}(\mu\pi)=0$, $i=0,1,...,m$
such that
\begin{equation}\label{eq-particular-solutions1}
\widetilde{\Delta} \Big( e^{-\gamma t}\widetilde{w}\Big)= e^{-\gamma t}\widetilde{h}
\quad\text{in }\mathbb R\times [0,\mu\pi].\end{equation}
Moreover, if $\widetilde{h} \in  C^{\infty}[0,\mu \pi]$,
then $\widetilde{w}\in   C^{\infty}[0,\mu \pi]$.
\end{lemma}

\begin{proof}
 A straightforward calculation yields $$\widetilde{\Delta}  \left( e^{-\gamma t}\widetilde{w}\right)=  e^{-\gamma t}\left(\frac{d^2 \widetilde{w}}{d \theta^2}+\gamma^2\widetilde{w}\right).$$
 We have the desired result by solving the ODE: $\frac{d^2 \widetilde{w}}{d \theta^2}+\gamma^2\widetilde{w}=\widetilde{h}$ with $\widetilde{w}(0)=\widetilde{w}(\mu\pi)=0$.
\end{proof}

\begin{lemma}\label{lemma-particular-solutions2}
Let $m$ be a nonnegative integer,  $\gamma$ be a positive constant,
and $\widetilde{h_0}, \widetilde{h_1}, \cdots, \widetilde{h_m}$ $\in C[0,\mu \pi]$.
Then, there exist $\widetilde{w_0}, \widetilde{w_1}, \cdots, \widetilde{w_m},\widetilde{ w_{m+1}}\in  C^2[0,\mu \pi]$ with $\widetilde{w_i}(0)=\widetilde{w_i}(\mu\pi)=0$, $i=0,1,...,m$
such that
\begin{equation}\label{eq-particular-solutions2}
\widetilde{\Delta} \Big(\sum_{j=0}^{m+1} t^je^{-\gamma t}\widetilde{w_j}\Big)=\sum_{j=0}^{m} t^je^{-\gamma t}\widetilde{h_j}
\quad\text{in }\mathbb R\times [0,\mu\pi].\end{equation}
Moreover, if $\widetilde{h_j} \in  C^{\infty}[0,\mu \pi]$ for any $j=0, 1, \cdots, m$,
then $\widetilde{w_j}\in   C^{\infty}[0,\mu \pi]$  for $j=0, 1, \cdots, m+1$.
\end{lemma}
The proof is basically the same as Lemma 7.3 in \cite{HJS}, we omit it.

\smallskip

We now  are ready to prove Theorem \ref{thm-v-1}.
\begin{proof}[Proof of Theorem \ref{thm-v-1}]

We first consider the case that
\begin{equation}\label{eq-special}\mathcal I_1\cap \mathcal I_2=\emptyset.\end{equation}
That is the case $\mu$ is an irrational number.

We always assume $t>T$ for some constant $T$ sufficiently large.

By \eqref{v-est-t}, it is easy to check
\begin{equation} \| e^{-2t}\widetilde{ F_0(v)} \|_{C^{\alpha}([\tau-1,\tau+1]\times[0,\mu\pi])  }\leq Ce^{-(2+2\alpha)\tau},\quad \forall \tau>T.
\end{equation}

{\it Step 1.} \eqref{add-1} holds for $m=1$.  We consider the following three cases.

{\it Case 1.1.} $2+2\alpha>\mu_1=\frac{1}{\mu}$.

Since $2+2\alpha<\frac 2\mu$,  by Lemma \ref{lemma-nonhomogeneous-linearized-eq}, we conclude that
there exists a constant $c$, $\widetilde{v_1}= ce^{-\mu_1t}\sin(\mu_1\theta)$ such that,
for any $t>T$,
\begin{equation*}
|\widetilde{v}-\widetilde{v_1} |\leq Ce^{-(2+2\alpha) t}.
\end{equation*}
Set $\widetilde{r_1}=\widetilde{v}-\widetilde{v_1}$. In the $x$-coordinates, by
\begin{equation}\label{eq-v} (1+v_{22})v_{11}+v_{22}-v_{12}v_{12}=0,\end{equation}
we have
$$(1+v_{22})r_{1,11}+r_{1,22}-v_{12}r_{1,12}=-v_{22}v_{1,11}+v_{12}v_{1,12}\triangleq h.$$
Then, by Schauder estimate, we have
\begin{align}\label{r1-est}\begin{split}
r^{2+\alpha}[D^2r_1 ]_{C^{\alpha}(R_{1,r})}
\leq  C(\mu,\alpha)\big(\|r_1\|_{L^{\infty}(R_{2,r})}+  r^{2+\alpha}[h]_{C^{\alpha}(R_{2,r})} + r^2\|h\|_{L^{\infty}(R_{2,r})}\big)
\end{split}\end{align}
where $R_{1,r}=\{x| \frac{r}{e}\leq |x|\leq er\}\bigcap V_{\mu}$, $R_{2,r}=\{x| \frac{r}{e^2}\leq |x|\leq e^2r\}\bigcap V_{\mu}$.

Hence, in $(t, \theta)$-coordinates, by interpolation inequality, there exists $\alpha_1>0$ such that
\begin{equation}\label{r1-est-t} \|\widetilde{r_1}\|_{C^{2, \alpha}([\tau-1,\tau+1]\times[0,\mu\pi])  }\leq Ce^{-(\mu_1+\alpha_1) \tau}.\end{equation}

{\it Case 1.2.}  If $2+2\alpha=\mu_1$.

For any $\varepsilon>0$ sufficiently small, we have
\begin{equation*} \| e^{-2t}\widetilde{ F_0(v)} \|_{C^{\alpha}([\tau-1,\tau+1]\times[0,\mu\pi])  }\leq Ce^{-(2+2\alpha-\varepsilon)\tau}.\end{equation*}
Lemma \ref{lemma-nonhomogeneous-linearized-eq}
implies
\begin{equation}\label{eq-U2-03-0b-Lip}
|\widetilde{v} |\leq Ce^{-(\mu_1-\varepsilon) t}.
\end{equation}
By a similar argument as in the proof \eqref{r1-est}, we have
\begin{equation*} \|\widetilde{v}\|_{C^{2, \alpha}([\tau-1,\tau+1]\times[0,\mu\pi])  }\leq Ce^{-(\mu_1-\varepsilon) \tau}
\end{equation*}
and
\begin{equation*} \| e^{-2t}\widetilde{ F_0(v)} \|_{C^{\alpha}([\tau-1,\tau+1]\times[0,\mu\pi])  }\leq Ce^{-(2\mu_1-2\varepsilon)\tau}.\end{equation*}
Choose $\varepsilon$ small enough such that  $2\mu_1-2\varepsilon>\mu_1$.

Then, by Case 1.1, there exists a constant $c$ and $\widetilde{v_1}= ce^{-\mu_1t}\sin (\mu_1\theta)$, such that
$\widetilde{r_1}=\widetilde{v}-\widetilde{v_1}$ satisfying
\begin{equation}\label{r1-est-t}
 \|\widetilde{r_1}\|_{C^{2, \alpha}([\tau-1,\tau+1]\times[0,\mu\pi])  }\leq Ce^{-(\mu_1+\alpha_1) \tau},
\end{equation}
for some $\alpha_1>0$.

{\it Case 1.3.}
 If $2+2\alpha<\mu_1$.

By Lemma \ref{lemma-nonhomogeneous-linearized-eq}, we have
\begin{equation*}
|\widetilde{v} |\leq Ce^{-(2+2\alpha) t}.
\end{equation*}

By a similar argument as in the proof \eqref{r1-est}, we have
\begin{equation*}
 \|\widetilde{v}\|_{C^{2, \alpha}([\tau-1,\tau+1]\times[0,\mu\pi])  }\leq Ce^{-(2+2\alpha) \tau},\end{equation*}
and hence
\begin{equation*}
 \| e^{-2t}\widetilde{ F_0(v)} \|_{C^{\alpha}([\tau-1,\tau+1]\times[0,\mu\pi])  }\leq Ce^{-(2+4\alpha)\tau}.\end{equation*}
Repeating the above argument finite times, we have
\begin{equation*} \|\widetilde{v}\|_{C^{2, \alpha}([\tau-1,\tau+1]\times[0,\mu\pi])  }\leq Ce^{-(2+2^{i}\alpha) \tau},\end{equation*}
and
\begin{equation*} \| e^{-2t}\widetilde{ F_0(v)} \|_{C^{\alpha}([\tau-1,\tau+1]\times[0,\mu\pi])  }\leq Ce^{-(2+2^{i+1}\alpha)\tau}
\end{equation*}
where $2+2^{i}\alpha< \mu_1\le 2+2^{i+1}\alpha$.

Now, we are in a similar situation as at the beginning of Case 1.1 or Case 1.2.
Hence, in this case, we also have estimate \eqref{r1-est-t}.

In summary, in all three cases,
there exists a constant $c$ and $\widetilde{v_1}= ce^{-\mu_1t}\sin(\mu_1 \theta)$, such that
$\widetilde{r_1}=\widetilde{v}-\widetilde{v_1}$ satisfies
\begin{equation} \label{est-r1}
\|\widetilde{r_1}\|_{C^{2, \alpha}([\tau-1,\tau+1]\times[0,\mu\pi])  }\leq Ce^{-(\mu_1+\alpha_1) \tau},
\end{equation}
for some $\alpha_1>0$. Also, we have $\widetilde{r_1}(t,0)=\widetilde{r_1}(t,\mu\pi)=0$.
Hence,
\begin{equation*} |\widetilde{r_1}|\leq  C\theta(\mu\pi-\theta)e^{-(\mu_1+\alpha_1) t}.\end{equation*}
This finishes the discussion of Step 1.
\smallskip

{\it Step 2.}
Assume the following estimates hold
\begin{equation}\label{induction-1}
\|\tilde v-\sum_{i=1}^m c_i(\theta) e^{-\mu_i t}\|_{C^{2,\alpha}([\tau-1,\tau+1]\times[0,\mu\pi])}\le C_m e^{-(\mu_m+\alpha_m)\tau}
\end{equation}
for some $\alpha_m>0$. Here we can also choose  $\alpha_m$ to be a number such that $\{\alpha_m+\mathcal I\}\bigcap\mathcal I =\emptyset$. By Step 1, \eqref{induction-1} is true for $m=1$. Now we improve \eqref{induction-1} to $m+1$.
Denote \begin{equation*}
\tilde r_m=\tilde v-\sum_{i=1}^m c_i(\theta) e^{-\mu_i t},\quad \tilde v_i=c_i(\theta) e^{-\mu_i t},\quad i=1,\cdots,m.
\end{equation*}
We will discuss the following two different situations.
\begin{itemize}
\item[(I).] $\mu_m+\alpha_m>\mu_{m+1}$.   Then we are done. Since we can choose $c_{m+1}=0$ and $\alpha_{m+1}=\mu_m+\alpha_m-\mu_{m+1}$.
\item[(II).] $\mu_m+\alpha_m<\mu_{m+1}$.
By the definition of $\widetilde r_m$ and \eqref{eq-in-t},  $\widetilde r_m$ satisfies
\begin{equation}\label{eqn-r-m}
\widetilde \Delta \widetilde r_m=e^{-2t}\widetilde {F_0(v)}+\sum_{i=1}^m \widetilde \Delta \widetilde v_i.
\end{equation}
By the decomposition of $\widetilde v$ and the definition of $F_0(v)$,  $e^{-2t}\widetilde {F_0(v)}$ is
the sum of the terms with following orders
\begin{equation}\label{induction-2}
e^{-(\mu_i+\mu_j-2)t},\	e^{-(\mu_i+\mu_m+\alpha_m-2)t},\	e^{-(2\mu_m+2\alpha_m-2)t},\	i,j=1,\cdots,m.
\end{equation}
Denote $$\mathcal J_m=\{\mu_i+\mu_j-2|i,j=1,\cdots,m\}\cup\{\mu_i+\mu_m+\alpha_m-2|i=1,\cdots,m\}\cup \{2\mu_m+2\alpha_m-2\}.
$$
Rearrange the elements in $\mathcal J_m$ as
\begin{equation*}
\gamma_{1}<\gamma_{2}<\cdots<\gamma_N.
\end{equation*}
Also by the assumptions, we also know $\mathcal J_m\cap \mathcal I_1=\emptyset$.
By induction assumption \eqref{induction-1}, one knows $|\widetilde \Delta \widetilde r_m|\le C_me^{-(\mu_m+\alpha_m)t}$.  Let $i_0$ be the index such that $\gamma_{i_0}<\mu_m+\alpha_m<\gamma_{i_0+1}$.  Then one knows
\begin{equation*}
\widetilde\Delta \widetilde r_m=\sum_{i=i_0+1}^N C_{i}(\theta,t) e^{-\gamma_i t}
\end{equation*}
for some functions $C_i(\theta,t)$ such that $\|C_i(\theta,t)\|_{C^{\alpha}}\le C_m$.
Consider the following cases.
\\ (II.1) $\gamma_{i_0+1}=\mu_i+\mu_j-2$ for some $i,j\in\{1,\cdots,m\}$.
\\ Then one knows $\gamma_{i_0+1}=\mu_{k}$ for some $k\ge i+j\geq m+1$.
\\  (II.1.1) If $\gamma_{i_0+1}=\mu_{m+1}\notin \mathcal I_1$. Then one has $C_{i_0+1}(\theta,t)=C_{i_0+1}(\theta)$.  By Lemma \ref{lemma-particular-solutions1},  there exists $\widetilde v_{m+1}=c_{m+1}(\theta)e^{-\gamma_{i_0+1}t}$ solves
\begin{equation}
\widetilde \Delta \widetilde v_{m+1}=C_{i_0+1}(\theta)e^{-\gamma_{i_0+1}t},\quad \widetilde v_{m+1}(0,t)=\widetilde v_{m+1}(\mu\pi,t)=0.
\end{equation}
Then $\widetilde r_{m+1}=\widetilde r_m-\widetilde v_{m+1}$ satisfies
\begin{equation*}
|\widetilde \Delta \widetilde r_{m+1}|\le C_{m+1}e^{-\gamma_{i_0+2}t}.
\end{equation*}
 (II.1.1.1) If $(\gamma_{i_0+1},\gamma_{i_0+2})\cap \mathcal I_1=\emptyset$, then by Lemma \ref{lemma-nonhomogeneous-linearized-eq} and Schauder estimates, one knows
\begin{equation*}
\|\widetilde r_{m+1}\|_{C^{2,\alpha}([\tau-1,\tau+1]\times[0,\mu\pi])}\le C_{m+1}e^{-\gamma_{i_0+2}\tau}.
\end{equation*}
  (II.1.1.2)
If $(\gamma_{i_0+1},\gamma_{i_0+2})\cap \mathcal I_1=\{\frac{l}{\mu},\cdots,\frac{l+l'}{\mu}\}$ for some $l\ge 1,l'\ge 0$. Then for $\varepsilon>0$ small enough, we can apply Lemma \ref{lemma-nonhomogeneous-linearized-eq} to $\widetilde r_{m+1}$ with $\gamma=\frac{l}{\mu}-\varepsilon$ and also Schauder estimates to get
\begin{equation*}
\|\widetilde r_{m+1}\|_{C^{2,\alpha}([\tau-1,\tau+1]\times[0,\mu\pi])}\le C_{m+1}e^{-(\frac{l}{\mu}-\varepsilon)\tau}.
\end{equation*}
This proves \eqref{induction-1} for $m+1$  for Case (II.1.1).
\\ (II.1.2) If $k\ge m+2$ and $\mu_{m+1}\notin \mathcal I_1$. Then for $\varepsilon>0$ small enough, we can apply Lemma \ref{lemma-nonhomogeneous-linearized-eq} to $\widetilde r_{m+1}=\widetilde r_m$ with $\gamma=\mu_{m+1}+\varepsilon$ and also Schauder estimates to get
\begin{equation*}
\|\widetilde r_{m+1}\|_{C^{2,\alpha}([\tau-1,\tau+1]\times[0,\mu\pi])}\le C_{m+1}e^{-(\mu_{m+1}+\varepsilon)\tau}.
\end{equation*}
\\(II.1.3)If $k\ge m+2$ and $\mu_{m+1}\in \mathcal I_1$. For $\varepsilon>0$ small, applying Lemma \ref{lemma-nonhomogeneous-linearized-eq} to $\widetilde r_{m}$ with $\gamma=\mu_{m+1}+\varepsilon$ yields that
\begin{equation*}
\|\widetilde r_m-c_{m+1}e^{-\mu_{m+1}t}\sin(\mu_{m+1}\theta)\|_{C^{2,\alpha}([\tau-1,\tau+1]\times[0,\mu\pi])}\le C_{m+1}e^{-(\mu_{m+1}+\varepsilon)\tau}.
\end{equation*}
Denote $\widetilde r_{m+1}=\widetilde r_m-c_{m+1}e^{-\mu_{m+1}t}\sin(\mu_{m+1}\theta)$.
\\ (II.2)  $\gamma_{i_0+1}=\mu_1+\mu_m+\alpha_m-2$.  Then it is easy to see that $\gamma_{i_0+1}>\mu_{m+1}$. Then we can proceed similar arguments as (II.1.2) and (II.1.3) to get \eqref{induction-1} holds for $m+1$.

\end{itemize}
In both cases (I) and (II), we proved \eqref{induction-1} holds for $m+1$. This proves Theorem \ref{thm-v-1}
for $\mathcal I_1\cap \mathcal I_2=\emptyset$ and in this case, $c_{i,j}=0$ if $j\neq 0$.
\par
{\it Step 3.}  We consider the case $\mathcal I_1\cap \mathcal I_2\ne \emptyset$. In fact, the arguments in Step  1 are exactly the same.  In Step 2, one needs to replace \eqref{induction-1} by
\begin{equation}\label{induction-1-1}
\|\tilde v-\sum_{i=1}^m \sum_{j=0}^{i-1}c_{i,j}(\theta) t^j e^{-\mu_i t}\|_{C^{2,\alpha}([\tau-1,\tau+1]\times[0,\mu\pi])}\le C_m e^{-(\mu_m+\alpha_m)\tau}.
\end{equation}
Similarly, we denote
 \begin{equation*}
\tilde r_m=\tilde v-\sum_{i=1}^m\sum_{j=0}^{i-1} c_{i,j}(\theta) t^j  e^{-\mu_i t},\quad \tilde v_i=\sum_{j=0}^{i-1} c_{i,j}(\theta)   t^j e^{-\mu_i t},\quad i=1,\cdots,m.
\end{equation*}
Then Case (I) in Step 2 is the same. We only need to consider Case (II).  In this case,  $ e^{-2t} \widetilde {F_0(v)}$ is
the sum of the terms with following orders
\begin{equation}\label{induction-2}
t^k e^{-(\mu_i+\mu_j-2)t},\		t^l e^{-(\mu_i+\mu_m+\alpha_m-2)t},\	e^{-(2\mu_m+2\alpha_m-2)t}
\end{equation}
for $i,j=1,\cdots,m,\	0\le k\le i+j-2,\	0\le l\le i-1.$
Also
\begin{equation*}
\widetilde \Delta \widetilde r_m=\sum_{i=i_0+1}^N C_i(\theta,t) t^{2m-2} e^{-\gamma_i t}
\end{equation*}
 for some functions $C_i(\theta,t)$ such that $\|C_i(\theta,t)\|_{C^{\alpha}}\le C_m$.
 The only difference happens in Case (II.1.1).  In this case, $\gamma_{i_0+1}=\mu_{m+1}$ and $\mu_{m+1}$ may also belong $\mathcal I_1$.  Then one has
 \begin{equation*}
C_{i_0+1}(\theta,t) t^{2m-2} e^{-\gamma_{i_0+1} t}=\sum_{i=0}^{m-1}\widetilde C_i(\theta)t^i e^{-\mu_{m+1}t}.
 \end{equation*}
  By Lemma \ref{lemma-particular-solutions2},  there exists $\widetilde v_{m+1}=\sum\limits_{i=0}^m c_{m+1,i}(\theta)t^i e^{-\mu_{m+1}t}$ solving
\begin{equation}
\widetilde \Delta \widetilde v_{m+1}=\sum_{i=0}^{m-1}\widetilde C_i(\theta)t^i e^{-\mu_{m+1}t},\quad \widetilde v_{m+1}(0,t)=\widetilde v_{m+1}(\mu\pi,t)=0.
\end{equation}
Then $\widetilde r_{m+1}=\widetilde r_m-\widetilde v_{m+1}$ satisfies
\begin{equation*}
|\widetilde \Delta \widetilde r_{m+1}|\le C_{m+1,\varepsilon}e^{-(\gamma_{i_0+2}-\varepsilon)t}
\end{equation*}
for any $\varepsilon>0$ small.  The rest discussion are the same as Step 2.
This completes the proof for Theorem \ref{thrm-expansion-1}.
\end{proof}
\smallskip
We now consider the case $\mu \in [1/2,1)$.
\begin{proof}[Proof for Theorem \ref{thrm-expansion-2}]
The proof is basically the same as the proof of Theorem 1. Note that in Step 1, \eqref{v-setup1} implies the constant $c_1$ in  the term $c_1e^{-\mu_1t}\sin(\mu_1 \theta)$ must be $0$ when we apply Lemma \ref{lemma-nonhomogeneous-linearized-eq}.  Then the following proof is basically the same,  We skip the details.
\end{proof}

\section{Regularity of Solutions in cones and polygons}

In this section, we study the obstructions for the higher global regularity of solutions to the Monge-Amp\`ere equation in convex cones and polygons.
We will see the asymptotic expansion derived in Section \ref{sec-expansion} is an effective tool for explain these obstructions.

\begin{theorem} \label{c10-neg-1}
Let $\mu\in (0,1/2)$ and
$u$ be a solution of \eqref{eq-MA-cone-eq}-\eqref{eq-MA-cone-boundary}
satisfying \eqref{v-setup1} and \eqref{v-setup2}. Suppose there exists $\delta>0$ such that
\begin{equation}
v\leq0 \quad\quad \text{in    }V_{\mu}\bigcap B_{\delta}(0).
\end{equation}
Then, either $v\equiv0$ or the constant $c_{1,0}$ in the expansion \eqref{expansion-1} of $u$ is negative.

\end{theorem}

\begin{proof}
If $v\equiv0$ in $V_{\mu}\bigcap B_{\delta}(0)$,  the analyticity of $v$ yields $v\equiv0$ in $V_{\mu}\bigcap B_{1}(0)$.  Otherwise, by the strong maximum principle,
$v<0 $ in $V_{\mu}\bigcap B_{\delta}(0)$. By the Hopf's lemma, we have
 $$\frac{\partial v}{\partial \theta}|_{r=\frac{\delta}{2},\theta=0}<0,$$
 and
 $$\frac{\partial v}{\partial \theta}|_{r=\frac{\delta}{2},\theta=\mu \pi}>0.$$
Set $h=|x|^{\frac 1\mu}\sin\frac{\theta}{\mu}$. Then, for $\varepsilon>0$ sufficiently small,
\begin{equation*}
v\leq - \varepsilon h\quad\quad \text{on    }V_{\mu}\bigcap \partial B_{\frac{\delta}{2}}(0).
\end{equation*}
Claim:
\begin{equation}\label{v-le}
v\leq - \varepsilon h\quad\quad \text{in    }V_{\mu}\bigcap B_{\frac{\delta}{2}}(0).
\end{equation}
Suppose not, then there exists $x_0\in V_{\mu}\bigcap B_{\frac{\delta}{2}}(0)$ such that
$$(v+\varepsilon h)(x_0)=\max_{x\in  \overline{ V_{\mu}\bigcap B_{\frac{\delta}{2}}(0})}[v-(-\varepsilon h)]  >0$$
Then, at $x_0$,
\begin{equation*}
\big[D^2 \big(\frac{|x|^2}{2}- \varepsilon h \big)\big]\geq \big[ D^2 \big(\frac{|x|^2}{2}+v\big)\big],
\end{equation*}
Hence, the matrix $\big[D^2\big(\frac{|x|^2}{2}- \varepsilon h\big) \big]$ is positive definite at $x_0$.
On the other hand, since $\Delta h=0$, one has
\begin{equation*}
\det D^2\big(\frac{|x|^2}{2}- \varepsilon h\big)=1+\varepsilon^2(h_{11}h_{22}-h_{12}^2)<1=\det D^2 \big(\frac{|x|^2}{2}+v\big),
\end{equation*}
which  yields a contradiction. Therefore, we have \eqref{v-le}.
Then the desired result follows directly by the expansion \eqref{expansion-1}.

\end{proof}

\begin{theorem}
Let $\mu\in (1/2,1)$ and
$u\in C^2 (\overline{V_{\mu}\bigcap  B_{1}(0)})$ be a solution of \eqref{eq-MA-cone-eq}-\eqref{eq-MA-cone-boundary} satisfying \eqref{v-setup2}. Suppose there exists $\delta>0$ such that
\begin{equation}
v\leq0 \quad\quad \text{in    }V_{\mu}\bigcap  B_{\delta}(0).
\end{equation}
Then, $v\equiv 0$.

\end{theorem}

\begin{proof}
If $v$ is not identical to $0$,
by similar arguments as the proof of Lemma \ref{c10-neg-1}, we have
\begin{equation*}
v\leq - \varepsilon h\quad\quad \text{on    }V_{\mu}\bigcap B_{\frac{\delta}{2}}(0)
\end{equation*}
for $\varepsilon>0$ sufficiently small, where $h=|x|^{\frac 1\mu}\sin\frac{\theta}{\mu}$.

This implies $u$ can't be $C^2$ up to the origin, which leads to a contradiction.

\end{proof}

Let $c\in(0,1)$ be a constant, $u\in C^2(V_{1/2})\bigcap  C(\overline{V_{1/2}})$ be a solution of
\begin{align}\label{eq-MA-c}
\det D^2 u&=c
\quad\quad\quad\text{   in } V_{1/2},\\
\label{eq-MA-c-boundary}u&=\frac{|x|^2}{2} \quad  \text{ on     }\partial V_{1/2}.
\end{align}
We adapt some notations from \cite{LeSavin2023}.
Denote
\begin{equation*}
P_{c}^{\pm}=\frac{|x|^2}{2}\pm\sqrt{1-c}x_1x_2.
\end{equation*}
According to [Theorem 1.3,\cite{LeSavin2023}], $u$ must be one of the following forms:
\begin{equation*}
P_{c}^{+}, \	P_{c}^{-},\	\lambda^2\overline{P}_c(\frac{x}{\lambda}),\	\lambda^2\underline{P}_c(\frac{x}{\lambda})
\end{equation*}
for some $\lambda\in (0,\infty)$, where $\overline{P}_c\in C^{2,\alpha}(\overline{V_{1/2}\bigcap B_1(0)})$ for some $\alpha\in (0,1)$ is a particular solution to \eqref{eq-MA-c}-\eqref{eq-MA-c-boundary} satisfying
$ P_{c}^{-}<\overline{P}_c< P_{c}^{+}$ and $\overline{P}_c=P_{c}^{+}+O(|x|^{2+\alpha})$  near $x=0$, $\underline{P}_c(x)$ has a conical
singularity at $x=0$.

Let $A_c^+$ be the affine transformation from $ \mathbb{R}^{2}$ to
$ \mathbb{R}^{2}$ given by the matrix

\begin{equation}
A_c^+:= \begin{pmatrix}
1 & -\frac{\sqrt{1-c}}{\sqrt{c}}\\[3pt]
 0 & \frac{1}{\sqrt{c}}
 \end{pmatrix},
\end{equation}
and denote $ V_{\mu^+_c}=(A_c^+)^{-1} V_{1/2}$. Then, $ V_{\mu^+_c}$ is an infinite cone with its vertex at the origin and its
open angle is $\mu^+_c\pi$, where $\mu^+_c=\frac{\arccos \sqrt{1-c} }{\pi} \in(0,1/2)$.

Then
\begin{equation*}
P_c^+\circ A_c^+=\frac{|x|^2}{2}\quad \text{in}\quad  V_{\mu^+_c}.
\end{equation*}
   Set $q=\frac{|x|^2}{2}$ and $q_c^+=q\circ A_c^+$. Then
   \begin{equation*}
   q_c^+=\frac{|x|^2}{2}\quad  \text{on}\quad \partial V_{\mu^+_c}.
   \end{equation*}
 For $u\in C^2(V_{1/2})\bigcap  C(\overline{V_{1/2}})$ solving \eqref{eq-MA-c}-\eqref{eq-MA-c-boundary},
set $u_c=u \circ A_c^+$. Then
\begin{align*}
\det D^2 u_c&= 1
\quad\quad\quad\text{   in } V_{\mu^+_c},\\
u_c&=\frac{|x|^2}{2} \quad  \text{ on     }\partial V_{\mu^+_c}.
\end{align*}

\begin{cor}\label{cor-reg-c}
Let $c\in (0,1)$ be a constant such that $\frac{1}{\mu^+_c}=\frac{\pi}{\arccos \sqrt{1-c} }$ is not an integer.
Then $\overline{P}_c \notin C^{\infty}(\overline{V_{1/2}})$.
\end{cor}
\begin{proof}
Set \begin{equation*}
u(x)=\overline{P}_c\circ A_c^+,\quad v(x)=u(x)-\frac 12 |x|^2.
\end{equation*}
Then by the property of $\overline{P_c}$ and the above discussion, one knows $v<0$.
From Theorem \ref{c10-neg-1},
the constant $c_{1,0}$ in the expansion \eqref{expansion-1} of $v$ is negative.
Hence, $u\notin C^{\infty}(\overline{V_{1/2}})$.

\end{proof}

In fact, the proof given above shows the optimal regularity of $\overline{P}_c $ near the origin is $C^{ \left[ \frac{1}{\mu^+_c}\right], \frac{1}{\mu^+_c}- \left[\frac{1}{\mu^+_c}\right]}$.

\begin{example}\label{eg-reg-c}
Let $\Omega=(0,1)^2$. Consider the Dirichlet problem
\begin{align}
\label{eq-eg-1}\det D^2 u_{c}&= c
\quad\quad\quad\text{   in } \Omega,\\
\label{eq-eg-bdy1}\ u_c&=\frac{|x|^2}{2} \quad  \text{ on     }\partial \Omega.
\end{align}
\end{example}
When $c=1$ the solution of \eqref{eq-eg-1}-\eqref{eq-eg-bdy1} is given by $u=\frac{|x|^2}{2}$. However, this solution is unstable in $C^{2,\alpha}$ space for any $\alpha\in (0,1)$.
By \cite{LeSavin2023}, for any constant $c\in(0,1)$, \eqref{eq-eg-1}-\eqref{eq-eg-bdy1} admits a unique solution $u_c$.
Moreover, $u_c\in C^{2,\alpha}$ for some $\alpha=\alpha(c)\in(0,1)$ and $ u_c= P_c^+ +O(|x|^{2+\alpha})$.
By the maximum principle, $\frac{|x|^2}{2}<u_c< P_c^+$ in $\Omega$.

A similar argument as in the proof of Corollary \ref{cor-reg-c} yields
the optimal regularity of $u_{c}$ near the origin is $C^{ \left[ \frac{1}{\mu^+_c}\right], \frac{1}{\mu^+_c}- \left[\frac{1}{\mu^+_c}\right]}$
when $\frac{\pi}{\arccos \sqrt{1-c} }$ is not an integer.

It is easy to see the for any fixed $\alpha\in(0,1)$,
 $u_c\notin C^{2,\alpha}$ for $c<1$ sufficient close to $1$.
\smallskip

Next, we turn our attention to the existence of classical solutions for the Dirichlet boundary problem of Monge-Amp\`ere equation on bounded polyhedra.

Let $\Omega$ be a bounded convex polyhedra in $\mathbb{R}^n$.  Let $u$ be a convex function solves
\begin{align}\label{eq-MA-dirichlet-eq}
\det D^2 u&=f
\quad\quad\quad\text{   in } \Omega,\\
\label{eq-MA-dirichlet-boundary}u&=\phi  \quad\quad\text{on }\partial \Omega.
\end{align}
where $f\in C^{\alpha}(\overline{\Omega})$, $f>0$ and $\phi\in C^{2,\alpha}(\overline{\Omega})$ for some $\alpha \in (0,1)$.

Le and Savin \cite{LeSavin2023} proved $u\in C^{2}(\overline{\Omega})$ under condition that there is a convex $C^{2}(\overline{\Omega})$ subsolution $\underline{u}$ to \eqref{eq-MA-dirichlet-eq} with $\underline{u}=\phi$ on $\partial\Omega$.
Meanwhile, they provide an example to show that such a result cannot hold when $n\geq3$.
However, the boundary values in the example they constructed do not satisfy the compatibility conditions at the vertices of the polyhedron. We will construct examples to illustrate that even if the boundary value satisfies the compatibility condition, \eqref{eq-MA-dirichlet-eq}-\eqref{eq-MA-dirichlet-boundary} does not necessarily admit a
$C^{2}(\overline{\Omega})$ convex solution when $n\geq3$.

\begin{example}\label{eg-reg-c}
Let $\Omega=(0,1)^3$ and $P=\frac{x_1^2+x_2^2+x_3^2}{2}-\frac{x_1x_2}{2}$. Then, $\underline{u}=P\in C^{\infty}(\overline{\Omega})$
is a convex solution of \eqref{eq-MA-dirichlet-eq}-\eqref{eq-MA-dirichlet-boundary} for $f=3/4$ and $\phi=P$.

Let $\widehat{f}$ be a smooth function in $\overline{\Omega}$ with $\widehat{f}=f$ on the vertices of $\overline{\Omega}$ and
$0<\widehat{f}<3/4$ on $\overline{\Omega}$ except for vertices.
We claim \eqref{eq-MA-dirichlet-eq}-\eqref{eq-MA-dirichlet-boundary} has no convex solution $u \in C^2(\overline{\Omega})$ for $f=\widehat{f}$ and $\phi=P$.
If not, assume that $u \in C^2(\overline{\Omega})$ is a convex solution for \eqref{eq-MA-dirichlet-eq}-\eqref{eq-MA-dirichlet-boundary}. By the strong maximum principle,
\begin{equation}\label{underlineu-u}
u> \underline{u} \quad\quad\text{   in } \Omega.
\end{equation}
A straightforward calculation yields $u_{12}(0,0,1/2)=-\sqrt{1- \widehat{ f}(0,0,1/2)}<-1/2$ since $u_{12}$ is continuous on $(0,0)\times[0,1]$.
Then, $u(t,t,1/2)< \underline{u}(t,t,1/2)$ for $t>0$ sufficiently small, which contradicts to \eqref{underlineu-u}.

\end{example}

In the example above, $\phi$ satisfies the compatibility condition at the vertex of $\Omega$ and $\det D^2 \underline{u}>\widehat{f} $ on $\overline{\Omega}$ except for vertices. However,
\eqref{eq-MA-dirichlet-eq}-\eqref{eq-MA-dirichlet-boundary} do not admit a convex solution $u \in C^2(\overline{\Omega})$.
\par The following example shows that if
 $\det D^2 \underline{u}=f$ on $\partial\Omega$, \eqref{eq-MA-dirichlet-eq}-\eqref{eq-MA-dirichlet-boundary}
may not admit a convex solution $u \in C^2(\overline{\Omega})$ either.

\begin{example}\label{eg-reg-c}
Let $\Omega=(0,4)^2\times (-2,2)$ and $P=\frac{x_1^2+x_2^2+x_3^2}{2}-\frac{x_1x_2}{2}$.  Then, $\underline{u}=P\in C^{\infty}(\overline{\Omega})$
is a convex solution of \eqref{eq-MA-dirichlet-eq}-\eqref{eq-MA-dirichlet-boundary} for $f=3/4$ and $\phi=P$.

Let $\widehat{f}$ be a smooth function in $\overline{\Omega}$ with $\widehat{f}=3/4$ on $\partial\Omega$,
$0<\widehat{f}\leq 3/4$ in $\Omega$ and $\widehat{f}$ is not identical to $3/4$ in $\Omega$.
Then, \eqref{eq-MA-dirichlet-eq}-\eqref{eq-MA-dirichlet-boundary} has no convex solution $u \in C^2(\overline{\Omega})$ for $f=\widehat{f}$ and $\phi=P$.
If not, assume that $u \in C^2(\overline{\Omega})$ is a convex solution. Set $w=u-\underline{u}$. Then, $w>0$ in $\Omega$, $w=0$ on $\partial\Omega$ and
$w(x_1,x_2,0)=o(x_1^2+x_2^2)$.

Set
\begin{equation*}
A:= \begin{pmatrix}
1 & \frac{1}{\sqrt{3}}\\[3pt]
 0 & \frac{2}{\sqrt{3}}
 \end{pmatrix}
 \quad\text{ and  }\quad
\widetilde{A}:= \begin{pmatrix}
1 & \frac{1}{\sqrt{3}} &0\\[3pt]
 0 & \frac{2}{\sqrt{3}}&0\\
 0& 0& 1
 \end{pmatrix} .
\end{equation*}
 Then $A^{-1}V_{2/3}=V_{\frac 12}$. Denote $\widetilde{\Omega}=\widetilde{A}^{-1}\Omega$,
$\widetilde{u}=u\circ \widetilde{A}$,
$\widetilde{\underline{u}}=\underline{u}\circ \widetilde{A}$ and $\widetilde{w}=w\circ \widetilde{A}\in C^2(\overline{\widetilde{\Omega}})$. Then,
$\widetilde{\underline{u}}=\frac{|x|^2}{2}$. $\widetilde{w}>0$ in $\widetilde{\Omega}$, $\widetilde{w}=0$ on $\partial\widetilde{\Omega}$ and
\begin{equation}\label{w=o}
\widetilde w(x_1,x_2,0)=o(x_1^2+x_2^2).
\end{equation}
Also, we have
\begin{equation*}
\sum_{i,j=1}^3 a_{ij}\widetilde{w}_{ij}\leq 0 \quad\quad\text{   in }\widetilde{ \Omega},
\end{equation*}
where
$$a_{ij}=\int_{0}^{1}cof\big(D^2 \widetilde{\underline{u}}+t(D^2 \widetilde u-D^2 \widetilde{\underline{u}})\big)dt $$
and cof$(M)$ is the co-factor matrix of $M$.
For $x'=(x_1,x_2)\in\overline{ V_{2/3}}\subseteq \mathbb{R}^2$, set $$h(x_1,x_2)=r'^{3/2}\sin\frac{3\theta}{2}$$
 where $r', \theta$ are the coordinate components of $x'$ in polar coordinates. Set
 $$v(x_1,x_2,x_3)=v(x_1,x_2,0)=h^{5/4}.$$
  A straightforward calculation yields
 $$ \Delta v=\frac{45}{64}(r')^{-1/8}(\sin \frac{3\theta}{2})^{-3/4}, $$
and $$|v_{ij}|\leq C\Delta v,\,\, i,j=1,2,3.$$
Note that $a_{ij}\in C(\overline{\widetilde{\Omega}})$ and $a_{ij}=\delta_{ij}$ on $\partial\widetilde{\Omega}$. Then, there exists $\delta>0$ small such that
\begin{equation*}
v-\frac{x_3^2}{100}\leq 0 \quad\quad\text{   on }(V_{2/3}\bigcap B_{\delta}(0))\times \{x_3=\pm1\},
\end{equation*}
and
\begin{equation*}
\sum_{i,j=1}^3 a_{ij}\left(v-\frac{x_3^2}{100}\right)_{ij}\geq 0 \quad\quad\text{   in }(V_{2/3}\bigcap B_{\delta}(0))\times [-1,1].
\end{equation*}
By the Hopf's lemma, we have $\varepsilon \left(v-\frac{x_3^2}{100}\right)\leq \widetilde{w}$  on $\partial \bigg((V_{2/3}\bigcap B_{\delta}(0))\times [-1,1]\bigg)$ for $\varepsilon>0$ sufficiently small.

Hence, by the maximum principle, $\widetilde{w}\geq \varepsilon\left(h-\frac{x_3^2}{100}\right)$ in $(V_{2/3}\bigcap B_{\delta}(0))\times [-1,1]$.
In particular, $\widetilde{w}(t,t,0)\geq c t^{15/8}$ for some constant $c>0$ and $t>0$ sufficiently small. This contradicts  \eqref{w=o}.

\end{example}


\end{document}